\documentclass[preprint,12pt]{elsarticle}

\usepackage{amsmath, amssymb, amsthm}
\usepackage{mathrsfs}
\usepackage{dsfont}
\usepackage{graphicx}
\usepackage{booktabs}
\usepackage{xcolor}
\definecolor{capaccent}{HTML}{765A86}
\usepackage{tikz}
\usepackage{float}
\usetikzlibrary{positioning, decorations.pathreplacing, calc, shapes.geometric}
\usepackage[margin=2.5cm]{geometry}
\usepackage{hyperref}
\usepackage{xurl}
\newtheorem{theorem}{Theorem}[section]
\newtheorem{proposition}[theorem]{Proposition}

\newtheorem{lemma}[theorem]{Lemma}

\newtheorem{conjecture}{Conjecture}[section]

\numberwithin{equation}{section}
\numberwithin{figure}{section}

\begin{document}

\begin{frontmatter}

\title{Asymptotic maximum spectral radius of nonregular graphs}

\author{Hangxi Cha}
\ead{2433949@tongji.edu.cn}
\author{Haiying Shan\corref{cor1}}
\ead{shan_haiying@tongji.edu.cn}
\affiliation{organization={
School of Mathematical Sciences, Key Laboratory of Intelligent Computing and Applications\\ (Ministry of Education), Tongji University},
            city={Shanghai},
            postcode={200092},
            country={P.R. China}}
\cortext[cor1]{Corresponding author}
\nonumnote{This work is partially supported by the National Natural Science Foundation of China (No. 12271182).}

\begin{abstract}
Let $\lambda_1(n,\Delta)$ be the maximum adjacency spectral radius among connected
nonregular simple graphs of order $n$ and maximum degree $\Delta$.
Using effective resistance bounds, explicit comparison graphs and a
one-dimensional Wirtinger inequality, we prove that, for every fixed integer
$\Delta\ge3$, $\lambda_1(n,\Delta)=\Delta-c_\Delta\pi^2/(4n^2)+O_\Delta(n^{-5/2})$
as $n\to\infty$,
where $c_\Delta=\Delta-1$ for odd $\Delta$ and $c_\Delta=2(\Delta-2)$ for even $\Delta$.
This proves the asymptotic conjecture posed by Liu
[J.\ Combin.\ Theory Ser.\ B 169 (2024), Conjecture 7.1].
\end{abstract}


\end{frontmatter}

\section{Introduction}\label{sec:introduction}

The adjacency spectral radius $\lambda_1(G)$ of a connected graph is at most its maximum
degree, with equality if and only if the graph is regular.
For a connected nonregular graph of order $n$ and maximum degree $\Delta$,
Stevanovi\'c~\cite{Stevanovic2004} obtained an early lower bound for
$\Delta-\lambda_1(G)$. Subsequent bounds incorporated the diameter and the number of
edges~\cite{Cioaba2007,CGN2007}. In particular, Cioab\u{a} proved that
$\Delta-\lambda_1(G)>1/(nD)$, where $D$ is the diameter of $G$.
These estimates lead to the question of the smallest possible difference
when the order and maximum degree are prescribed.

For integers $\Delta\ge3$ and $n\ge \Delta+1$, let $\mathcal G^*(n,\Delta)$ be the set of
connected nonregular simple graphs of order $n$ and maximum degree exactly
$\Delta$. Define
\[
\lambda_1(n,\Delta)=\max\{\lambda_1(G):G\in\mathcal G^*(n,\Delta)\}.
\]
Liu, Shen and Wang~\cite{LSW2007} studied the order of magnitude of
$\Delta-\lambda_1(n,\Delta)$ and conjectured that, for every fixed integer
$\Delta\ge3$,
\[
\lim_{n\to\infty}n^2\bigl(\Delta-\lambda_1(n,\Delta)\bigr)=(\Delta-1)\pi^2.
\]
Liu~\cite{Liu2024} disproved this conjecture for every $\Delta\ge3$,
determined the limits for $\Delta=3,4$, and proposed the following conjecture.
\begin{conjecture}[Liu {\cite[Conjecture~7.1]{Liu2024}}]
For every fixed integer $\Delta\ge3$,
\[
\lim_{n\to\infty}n^2\bigl(\Delta-\lambda_1(n,\Delta)\bigr)
=\begin{cases}
(\Delta-1)\pi^2/4,&\Delta\text{ odd},\\
(\Delta-2)\pi^2/2,&\Delta\text{ even}.
\end{cases}
\]
\end{conjecture}

We prove the following quantitative version of the conjecture.
Set
\[
c_\Delta=\begin{cases}\Delta-1,&\Delta\text{ odd},\\2(\Delta-2),&\Delta\text{ even}.
\end{cases}
\]
\begin{theorem}\label{thm:main}
For every fixed integer $\Delta\ge3$,
\[
\lambda_1(n,\Delta)=\Delta-\frac{c_\Delta\pi^2}{4n^2}+O_\Delta(n^{-5/2})
\qquad(n\to\infty).
\]
\end{theorem}

In Section~\ref{sec:resistance}, we derive resistance bounds using the
resistance theorem of Abdi and Ghorbani~\cite{AGI} for odd \(\Delta\) and their
algebraic connectivity estimate~\cite{AGII} for even \(\Delta\).
We then extend these bounds to networks with more general terminal degrees.
Section~\ref{sec:proof} completes the proof of Theorem~\ref{thm:main}.
The upper bound for $\Delta-\lambda_1(n,\Delta)$ follows from Rayleigh
quotients of comparison graphs, and the lower bound from the resistance bounds,
piecewise linear interpolation and Wirtinger's inequality.

\section{Preliminaries}\label{sec:preliminaries}

Graphs are finite, undirected and simple unless stated otherwise.
Write $d_G(v)$ for the degree of $v$,
$e(G)$ for the number of edges, $A(G)$ for the adjacency matrix, and
$L_G$ for the Laplacian matrix.
For a connected graph \(G\) with at least two vertices, let \(\nu_2(G)\)
denote its algebraic connectivity,
the second-smallest eigenvalue of \(L_G\).
In a multigraph, degrees and edge counts
include multiplicity.
The vector $\mathbf1$ has all coordinates equal to one. For a real vector
$x=(x_v)_{v\in V(G)}$, write $\|x\|_2^2=\sum_{v\in V(G)}x_v^2$.
In asymptotic estimates, $\Delta$ is fixed. Subscripts in $O$-notation
indicate the parameters on which the implicit constants may depend.

For a simple graph \(G\) with maximum degree at most \(\Delta\), we call
\(\Delta-d_G(v)\) the \textit{degree deficit} of \(v\) relative to \(\Delta\).
For every \(y\in\mathbb R^{V(G)}\), we have
\begin{equation}
y^{\mathsf T}(\Delta I-A(G))y
=\sum_{uv\in E(G)}(y_u-y_v)^2
 +\sum_{v\in V(G)}(\Delta-d_G(v))y_v^2.                                \label{eq:deficit}
\end{equation}
The least eigenvalue of \(\Delta I-A(G)\) is \(\Delta-\lambda_1(G)\), so any nonzero
\(y\) gives an upper bound for \(\Delta-\lambda_1(G)\) through its
Rayleigh quotient.

In this paper, a network is a connected loopless multigraph with unit edge resistances.
Let $H$ be such a network and let $a,b \in V(H)$ be two distinct
\textit{terminals}. The remaining
vertices are termed \textit{nonterminals}. For any function
$x \in \mathbb{R}^{V(H)}$, its Dirichlet energy is defined as
\[
    \mathcal{E}_H(x) = \sum_{uv \in E(H)} (x_u - x_v)^2,
\]
where parallel edges are counted separately.  Let
$\Omega_{a,b} = \{ x \in \mathbb{R}^{V(H)} \mid x_a = 0,\, x_b = 1 \}$
be the set of functions with fixed boundary values at the terminals.
The effective resistance between $a$ and $b$ is given by
\[
R_H(a,b) = \left( \min_{x \in \Omega_{a,b}} \mathcal{E}_H(x) \right)^{-1}.
\]

A vertex function \(x\) is harmonic at a nonterminal \(v\) if
\(\sum_{vw\in E(H)}(x_v-x_w)=0\), with parallel edges counted separately.
By the maximum principle~\cite[Section~1.3.2]{DoyleSnell},
if \(x\) is harmonic at every nonterminal, then
\(\min\{x_a,x_b\}\le x_v\le\max\{x_a,x_b\}\) for every \(v\in V(H)\).

By the Dirichlet principle~\cite[Section~1.3]{DoyleSnell}, there is a unique
function \(u\) satisfying
\begin{equation}
u\in\Omega_{a,b},\qquad
\mathcal E_H(u)=\min_{x\in\Omega_{a,b}}\mathcal E_H(x)
=\frac{1}{R_H(a,b)}.                                            \label{eq:harmonic}
\end{equation}
This function is harmonic at every nonterminal, so the maximum principle
gives \(0\le u_v\le1\) for every \(v\in V(H)\).
The Dirichlet principle also gives, for every \(x\in\mathbb R^{V(H)}\),
\begin{equation}
\mathcal E_H(x)\ge\frac{(x_b-x_a)^2}{R_H(a,b)}.                     \label{eq:resistance}
\end{equation}
Let \(G\) be a connected simple graph with distinct terminals \(a,b\),
and let \(G^{(1)},\ldots,G^{(q)}\), \(q\ge1\), be disjoint copies of \(G\),
with terminals \(a_i,b_i\) in copy \(i\). Adding the edges
\(b_i a_{i+1}\), \(1\le i<q\), gives a chain with terminals \(a_1,b_q\).
Each added edge is a bridge of this chain.
We shall also use chains obtained by identifying \(b_i\) with
\(a_{i+1}\), \(1\le i<q\). In this case, denote the identified vertex by
\(v_i\), and set \(v_0=a_1\), \(v_q=b_q\). The \(i\)-th copy then has
terminals \(v_{i-1},v_i\).

For even \(\Delta\ge4\), suppose that \(d_G(a)=2\), \(d_G(b)=\Delta-2\),
and every nonterminal has degree \(\Delta\). Form a chain of \(k\ge1\)
copies of \(G\) by identifying \(b_i\) with \(a_{i+1}\), \(1\le i<k\).
Attach a disjoint copy of \(K_{\Delta+1}-M\) at \(v_0\), where \(M\)
is a matching of \((\Delta-2)/2\) edges, by joining \(v_0\) to every vertex
incident with an edge of \(M\). At \(v_k\), attach a disjoint copy of
\(K_{\Delta+1}-e\) by joining \(v_k\) to the two vertices incident with \(e\).
The resulting graph, denoted
by \(G_k(G)\), is connected, simple and \(\Delta\)-regular, of order
\(k(|V(G)|-1)+2\Delta+3\). We call the copies attached at \(v_0\) and
\(v_k\) the left and right caps, respectively.
Figure~\ref{fig:regular-chain} illustrates the construction for \(\Delta=6\).

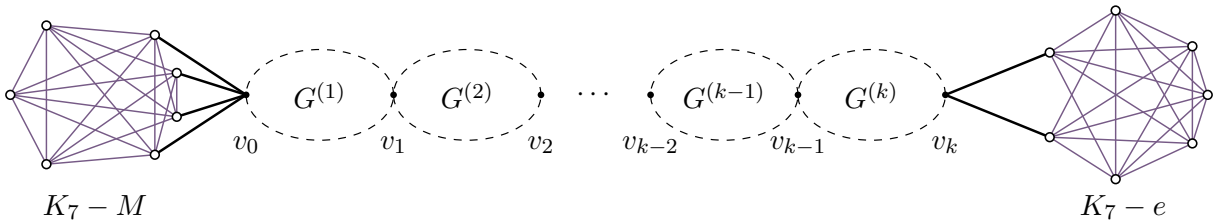
\begin{figure}[H]
\centering
\begin{tikzpicture}[x=10mm,y=10mm,line width=.45pt,font=\small,
  cap edge/.style={draw=capaccent,line width=.55pt},
  cap vertex/.style={draw=black,fill=white,line width=.6pt},
  attachment/.style={draw=black,line width=.95pt}]
\coordinate (v0) at (1.8,0);
\coordinate (v1) at (3.75,0);
\coordinate (v2) at (5.7,0);
\coordinate (vkm2) at (7.15,0);
\coordinate (vkm1) at (9.1,0);
\coordinate (vk) at (11.05,0);
\coordinate (cL0) at (0.592,0.792);
\coordinate (cL1) at (0.8818,0.2899);
\coordinate (cL2) at (0.8818,-0.2899);
\coordinate (cL3) at (0.592,-0.792);
\coordinate (cL4) at (-0.8424,-0.9175);
\coordinate (cL5) at (-1.32,0);
\coordinate (cL6) at (-0.8424,0.9175);
\coordinate (cR0) at (12.4301,0.56);
\coordinate (cR1) at (12.4301,-0.56);
\coordinate (cR2) at (13.3024,-1.1157);
\coordinate (cR3) at (14.3175,-0.6424);
\coordinate (cR4) at (14.52,0);
\coordinate (cR5) at (14.3175,0.6424);
\coordinate (cR6) at (13.3024,1.1157);
\draw[cap edge] (cL0) -- (cL2);
\draw[cap edge] (cL0) -- (cL3);
\draw[cap edge] (cL0) -- (cL4);
\draw[cap edge] (cL0) -- (cL5);
\draw[cap edge] (cL0) -- (cL6);
\draw[cap edge] (cL1) -- (cL2);
\draw[cap edge] (cL1) -- (cL3);
\draw[cap edge] (cL1) -- (cL4);
\draw[cap edge] (cL1) -- (cL5);
\draw[cap edge] (cL1) -- (cL6);
\draw[cap edge] (cL2) -- (cL4);
\draw[cap edge] (cL2) -- (cL5);
\draw[cap edge] (cL2) -- (cL6);
\draw[cap edge] (cL3) -- (cL4);
\draw[cap edge] (cL3) -- (cL5);
\draw[cap edge] (cL3) -- (cL6);
\draw[cap edge] (cL4) -- (cL5);
\draw[cap edge] (cL4) -- (cL6);
\draw[cap edge] (cL5) -- (cL6);
\draw[cap edge] (cR0) -- (cR2);
\draw[cap edge] (cR0) -- (cR3);
\draw[cap edge] (cR0) -- (cR4);
\draw[cap edge] (cR0) -- (cR5);
\draw[cap edge] (cR0) -- (cR6);
\draw[cap edge] (cR1) -- (cR2);
\draw[cap edge] (cR1) -- (cR3);
\draw[cap edge] (cR1) -- (cR4);
\draw[cap edge] (cR1) -- (cR5);
\draw[cap edge] (cR1) -- (cR6);
\draw[cap edge] (cR2) -- (cR3);
\draw[cap edge] (cR2) -- (cR4);
\draw[cap edge] (cR2) -- (cR5);
\draw[cap edge] (cR2) -- (cR6);
\draw[cap edge] (cR3) -- (cR4);
\draw[cap edge] (cR3) -- (cR5);
\draw[cap edge] (cR3) -- (cR6);
\draw[cap edge] (cR4) -- (cR5);
\draw[cap edge] (cR4) -- (cR6);
\draw[cap edge] (cR5) -- (cR6);
\draw[attachment] (v0) -- (cL0);
\draw[attachment] (v0) -- (cL1);
\draw[attachment] (v0) -- (cL2);
\draw[attachment] (v0) -- (cL3);
\draw[attachment] (vk) -- (cR0);
\draw[attachment] (vk) -- (cR1);
\foreach \v in {v0,v1,v2,vkm2,vkm1,vk}
  \fill (\v) circle[radius=1.25pt];
\foreach \v in {cL0,cL1,cL2,cL3,cL4,cL5,cL6,cR0,cR1,cR2,cR3,cR4,cR5,cR6}
  \draw[cap vertex] (\v) circle[radius=1.65pt];
\draw[dashed] (2.775,0) ellipse[x radius=.975,y radius=.6];
\node at (2.775,0) {$G^{(1)}$};
\draw[dashed] (4.725,0) ellipse[x radius=.975,y radius=.6];
\node at (4.725,0) {$G^{(2)}$};
\draw[dashed] (8.125,0) ellipse[x radius=.975,y radius=.6];
\node at (8.125,0) {$G^{(k-1)}$};
\draw[dashed] (10.075,0) ellipse[x radius=.975,y radius=.6];
\node at (10.075,0) {$G^{(k)}$};
\node at (6.425,0) {$\cdots$};
\node[below=4mm] at (v0) {$v_{0}$};
\node[below=4mm] at (v1) {$v_{1}$};
\node[below=4mm] at (v2) {$v_{2}$};
\node[below=4mm] at (vkm2) {$v_{k-2}$};
\node[below=4mm] at (vkm1) {$v_{k-1}$};
\node[below=4mm] at (vk) {$v_{k}$};
\node at (-.2,-1.48) {$K_7-M$};
\node at (13.4,-1.48) {$K_7-e$};
\end{tikzpicture}
\caption{The graph $G_k(G)$ for $\Delta=6$.}
\label{fig:regular-chain}
\end{figure}

For \(d\ge3\), put \(L_d=K_{d+1}-ab\), with terminals \(a,b\).
An \(L_d\) chain is obtained from disjoint copies of \(L_d\)
by adding the edges \(b_i a_{i+1}\) between consecutive copies.
For pairwise disjoint graphs \(G_1,\ldots,G_t\), their sequential join
\(G_1+\cdots+G_t\) is obtained from their disjoint union by adding
all edges between \(G_i\) and \(G_{i+1}\) for \(1\le i<t\).
For even \(\Delta\ge4\) and \(k\in\{2,\Delta-2\}\), define
\[
B_{\Delta,k}=K_1+K_k+K_{\Delta-k}+K_1.
\]
The vertices in the first and last \(K_1\) are its left and right terminals.
A \(B_{\Delta,k}\) chain is obtained from disjoint copies by
identifying the right terminal of each copy with the left terminal of the next.

For integers \(\Delta\ge3\), \(m\ge1\) and \(0\le r\le\Delta\), define
\(G_{\Delta,m,r}\) as follows. If \(\Delta\) is odd, start with an
\(L_\Delta\) chain of \(m\) copies, with terminals \(a_j,b_j\) in copy
\(j\). Take a disjoint graph on \(\Delta+2\) vertices \(c_0,\ldots,c_{\Delta+1}\),
obtained from \(K_{\Delta+2}\) by deleting \(c_0c_1,c_0c_2\) and a perfect
matching on \(c_3,\ldots,c_{\Delta+1}\), and attach it by the edge
\(b_mc_0\). Call this attached graph the right cap. If \(r>0\), add a path
\(p_1\cdots p_r\) on new vertices
and the edge \(a_1p_1\). The vertex \(c_0\) has degree \(\Delta-1\)
before attachment, and every other cap vertex has degree \(\Delta\).
The cap is connected: otherwise its minimum degree \(\Delta-1\) would
force each component to have at least \(\Delta\) vertices, contradicting
\(2\Delta>\Delta+2\).

If \(\Delta\) is even, start with a \(B_{\Delta,2}\) chain of \(m\)
copies, with terminals \(v_{j-1},v_j\) in copy \(j\). Attach a disjoint copy of
\(L_\Delta=K_{\Delta+1}-uv\) at \(v_m\) by the edges \(v_mu,v_mv\),
and call it the right cap. If \(r>0\), add a path \(p_1\cdots p_r\) on new vertices
and the edge
\(v_0p_1\). Every shared vertex, every cap vertex and \(v_m\) has degree
\(\Delta\).

In both cases, \(G_{\Delta,m,r}\) is connected and simple, with order
\(n=(\Delta+1)m+\Delta+2+r\) and maximum degree \(\Delta\).
When \(r=0\), its left terminal has degree \(\Delta-1\) in the odd case
and degree two in the even case; when \(r>0\), \(p_r\) has degree one.
Hence \(G_{\Delta,m,r}\in\mathcal G^*(n,\Delta)\).
Figure~\ref{fig:comparison-graphs} shows two examples.

\begin{figure}[H]
\centering
\begin{tikzpicture}[x=13mm,y=10mm,line width=.45pt,font=\small,
  cap edge/.style={draw=capaccent,line width=.55pt},
  cap vertex/.style={draw=black,fill=white,line width=.6pt},
  attachment/.style={draw=black,line width=.95pt}]
\begin{scope}
\coordinate (p1) at (0,0);
\coordinate (a1) at (1.1,0);
\coordinate (x1) at (2.05,0.58);
\coordinate (y1) at (2.05,-0.58);
\coordinate (b1) at (3.05,0);
\coordinate (a2) at (3.95,0);
\coordinate (x2) at (4.9,0.58);
\coordinate (y2) at (4.9,-0.58);
\coordinate (b2) at (5.9,0);
\coordinate (c0) at (6.8,0);
\coordinate (c3) at (7.8,0.58);
\coordinate (c4) at (7.8,-0.58);
\coordinate (c1) at (8.8,0.9);
\coordinate (c2) at (8.8,-0.9);
\draw (a1) -- (x1);
\draw (a1) -- (y1);
\draw (b1) -- (x1);
\draw (b1) -- (y1);
\draw (x1) -- (y1);
\draw (a2) -- (x2);
\draw (a2) -- (y2);
\draw (b2) -- (x2);
\draw (b2) -- (y2);
\draw (x2) -- (y2);
\draw[cap edge] (c0) -- (c3);
\draw[cap edge] (c0) -- (c4);
\draw[cap edge] (c1) -- (c2);
\draw[cap edge] (c1) -- (c3);
\draw[cap edge] (c1) -- (c4);
\draw[cap edge] (c2) -- (c3);
\draw[cap edge] (c2) -- (c4);
\draw (p1) -- (a1);
\draw (b1) -- (a2);
\draw[attachment] (b2) -- (c0);
\foreach \v in {p1,a1,x1,y1,b1,a2,x2,y2,b2} \fill (\v) circle[radius=1.25pt];
\foreach \v in {c0,c3,c4,c1,c2} \draw[cap vertex] (\v) circle[radius=1.65pt];
\node at (4.5,-2.0) {(a) $\Delta=3$, $m=2$, $r=1$};
\node[below=3pt] at (p1) {$p_1$};
\node[below=3pt] at (a1) {$a_1$};
\node[below=3pt] at (b1) {$b_1$};
\node[below=3pt] at (a2) {$a_2$};
\node[below=3pt] at (b2) {$b_2$};
\node[below=3pt] at (c0) {$c_0$};
\node[above=3pt] at (c1) {$c_1$};
\node[below=3pt] at (c2) {$c_2$};
\node[above=3pt] at (c3) {$c_3$};
\node[below=3pt] at (c4) {$c_4$};
\end{scope}
\begin{scope}[yshift=-42mm]
\coordinate (p2) at (0,0);
\coordinate (p1) at (0.65,0);
\coordinate (v0) at (1.3,0);
\coordinate (x11) at (2.1,0.58);
\coordinate (x12) at (2.1,-0.58);
\coordinate (y11) at (3.1,0.58);
\coordinate (y12) at (3.1,-0.58);
\coordinate (v1) at (3.9,0);
\coordinate (x21) at (4.7,0.58);
\coordinate (x22) at (4.7,-0.58);
\coordinate (y21) at (5.7,0.58);
\coordinate (y22) at (5.7,-0.58);
\coordinate (v2) at (6.5,0);
\coordinate (u) at (7.4,0.58);
\coordinate (v) at (7.4,-0.58);
\coordinate (cx) at (9.0,0);
\coordinate (cy) at (8.4,0.9);
\coordinate (cz) at (8.4,-0.9);
\draw (x11) -- (x12);
\draw (x11) -- (y11);
\draw (x11) -- (y12);
\draw (x12) -- (y11);
\draw (x12) -- (y12);
\draw (y11) -- (y12);
\draw (v0) -- (x11);
\draw (v0) -- (x12);
\draw (v1) -- (y11);
\draw (v1) -- (y12);
\draw (x21) -- (x22);
\draw (x21) -- (y21);
\draw (x21) -- (y22);
\draw (x22) -- (y21);
\draw (x22) -- (y22);
\draw (y21) -- (y22);
\draw (v1) -- (x21);
\draw (v1) -- (x22);
\draw (v2) -- (y21);
\draw (v2) -- (y22);
\draw[cap edge] (u) -- (cx);
\draw[cap edge] (u) -- (cy);
\draw[cap edge] (u) -- (cz);
\draw[cap edge] (v) -- (cx);
\draw[cap edge] (v) -- (cy);
\draw[cap edge] (v) -- (cz);
\draw[cap edge] (cx) -- (cy);
\draw[cap edge] (cx) -- (cz);
\draw[cap edge] (cy) -- (cz);
\draw[attachment] (v2) -- (u);
\draw[attachment] (v2) -- (v);
\draw (p2) -- (p1);
\draw (p1) -- (v0);
\foreach \v in {p2,p1,v0,x11,x12,y11,y12,v1,x21,x22,y21,y22,v2} \fill (\v) circle[radius=1.25pt];
\foreach \v in {u,v,cx,cy,cz} \draw[cap vertex] (\v) circle[radius=1.65pt];
\node at (4.5,-2.0) {(b) $\Delta=4$, $m=2$, $r=2$};
\node[below=3pt] at (p2) {$p_2$};
\node[below=3pt] at (p1) {$p_1$};
\node[below=3pt] at (v0) {$v_0$};
\node[below=3pt] at (v1) {$v_1$};
\node[below=3pt] at (v2) {$v_2$};
\node[above=3pt] at (u) {$u$};
\node[below=3pt] at (v) {$v$};
\end{scope}
\end{tikzpicture}
\caption{The graphs $G_{\Delta,m,r}$ for $\Delta=3,4$.}
\label{fig:comparison-graphs}
\end{figure}

We use the following resistance bound for odd \(\Delta\).
\begin{lemma}[{\cite[Theorem~5.3]{AGI}}]
\label{lem:AGI-resistance}
Let \(d\ge3\) be an integer and let \(G\) be a connected simple graph of order \(n\),
with distinct terminals \(a,b\). Suppose that \(d_G(a),d_G(b)\ge d-1\)
and every nonterminal has degree at least \(d\). Then
\[
R_G(a,b)\le\frac n{d-1}-1.
\]
Equality holds if and only if \(n\) is a multiple of \(d+1\) and
\(G\) is an \(L_d\) chain with terminals \(a,b\).
\end{lemma}

For even \(\Delta\), we use an algebraic connectivity estimate.
\begin{lemma}[{\cite[Corollary~1.2]{AGII}}]
\label{lem:even-connectivity}
For every fixed even integer \(\Delta\ge4\), the minimum algebraic
connectivity among connected simple \(\Delta\)-regular graphs of order
\(n\) is
\[
\frac{2(\Delta-2)\pi^2}{n^2}+O_\Delta(n^{-3})
\qquad(n\to\infty).
\]
\end{lemma}

\section{Resistance bounds}\label{sec:resistance}

We first consider prescribed terminal degrees.

\begin{lemma}\label{lem:terminal} Let \(G\) be a connected simple graph of order \(n\),
with distinct terminals \(a,b\), and suppose that every nonterminal has degree \(\Delta\).
For odd \(\Delta\ge3\), if the terminal degrees are \(\Delta-1,\Delta-1\), then
\begin{equation}
R_G(a,b)\le\frac n{\Delta-1}-1.                                        \label{eq:odd-resistance}
\end{equation}
For even \(\Delta\ge4\), if \(d_G(a)=2\) and \(d_G(b)=\Delta-2\), then
\begin{equation}
R_G(a,b)\le\frac{n-1}{2(\Delta-2)}.                                   \label{eq:even-resistance}
\end{equation}
\end{lemma}

\begin{proof}
For odd \(\Delta\), apply Lemma~\ref{lem:AGI-resistance} with \(d=\Delta\).

Now let \(\Delta\ge4\) be even. Set \(R=R_G(a,b)>0\) and \(m=n-1\), and let
\(u\in\mathbb R^{V(G)}\) minimize \(\mathcal E_G\) subject to \(u_a=0\), \(u_b=1\).
For each integer \(k\ge2\), let \(G_k=G_k(G)\) be the graph defined in
Section~\ref{sec:preliminaries}. Its \(i\)-th copy \(G^{(i)}\) has terminals
\(v_{i-1},v_i\).
Its order is \(n_k=km+2\Delta+3\).

\begin{samepage}
For \(1\le i\le k\), let \(u^{(i)}\in\mathbb R^{V(G^{(i)})}\)
be the copy of \(u\) on \(G^{(i)}\), so that
\(u^{(i)}_{v_{i-1}}=0\) and \(u^{(i)}_{v_i}=1\).
Put \(T_i=\cos(\pi i/k)\) for \(0\le i\le k\).
Define \(f\in\mathbb R^{V(G_k)}\) by
\[
f_v=T_{i-1}+(T_i-T_{i-1})u^{(i)}_v
\qquad (v\in V(G^{(i)}),\ 1\le i\le k),
\]
and set \(f_v=1\) on the left cap and \(f_v=-1\) on the right cap.
At each shared vertex \(v_i\), both formulas give \(T_i\),
so \(f\) is well-defined. Since \(T_0=1\) and \(T_k=-1\), each edge \(xy\)
inside a cap or joining a cap to the chain contributes \((f_x-f_y)^2=0\)
to \(\mathcal E_{G_k}(f)\).
\end{samepage}

By the definition of \(f\) and \eqref{eq:harmonic}, the edges of
\(G^{(i)}\) contribute \((T_i-T_{i-1})^2\mathcal E_G(u)
=(T_i-T_{i-1})^2/R\) to \(\mathcal E_{G_k}(f)\).
By the Cauchy--Schwarz inequality,
\[
(T_i-T_{i-1})^2
=\left(\int_{(i-1)/k}^{i/k}\pi\sin(\pi t)\,dt\right)^2
\le\frac1k\int_{(i-1)/k}^{i/k}\pi^2\sin^2(\pi t)\,dt.
\]
Summing over the copies, we obtain
\[
\mathcal E_{G_k}(f)=\frac1R\sum_{i=1}^k(T_i-T_{i-1})^2
\le\frac1{Rk}\int_0^1\pi^2\sin^2(\pi t)\,dt
=\frac{\pi^2}{2Rk}.
\]
We next estimate \(\sum_{v\in V(G_k)}f_v\) and \(\|f\|_2^2\). The sets \(V(G^{(i)})\setminus\{v_i\}\), \(1\le i\le k\),
are disjoint and have \(m\) vertices each. The remaining vertices are
\(v_k\) and the \(2(\Delta+1)\) cap vertices. For a vertex \(v\) in the \(i\)-th set, the maximum
principle gives \(0\le u^{(i)}_v\le1\). By the definition of \(f\), this
implies \(f_v\in[T_i,T_{i-1}]\subseteq[-1,1]\), and hence
\[
|f_v-T_{i-1}|\le |T_i-T_{i-1}|,\qquad
|f_v^2-T_{i-1}^2|\le 2|T_i-T_{i-1}|.
\]
Since \(\sum_{i=1}^k|T_i-T_{i-1}|=2\), while
\(\sum_{i=1}^kT_{i-1}=1\) and \(\sum_{i=1}^kT_{i-1}^2=k/2\), we obtain
\[
\left|\sum_{v\in V(G_k)}f_v\right|\le3m+2\Delta+3,\qquad
\left|\|f\|_2^2-\frac{mk}{2}\right|\le4m+2\Delta+3.
\]
Set \(z=f-n_k^{-1}(\sum_{v\in V(G_k)} f_v)\mathbf1\). Then
\(z\) has mean zero, \(\mathcal E_{G_k}(z)=\mathcal E_{G_k}(f)\), and
\[
\frac{\|z\|_2^2}{k}
=\frac{\|f\|_2^2}{k}-\frac1{kn_k}
 \left(\sum_{v\in V(G_k)}f_v\right)^2
\longrightarrow\frac m2\qquad(k\to\infty).
\]
Since \(f\) is nonconstant, \(z\ne0\), and the Rayleigh principle gives
\[
\nu_2(G_k)\le\frac{\mathcal E_{G_k}(z)}{\|z\|_2^2}
\le\frac{\pi^2}{2Rk\|z\|_2^2}.
\]
Since \(n_k/k\to m\), it follows that
\begin{equation}
\limsup_{k\to\infty}n_k^2\nu_2(G_k)\le \frac{m\pi^2}{R}.          \label{eq:regular-limit}
\end{equation}
By Lemma~\ref{lem:even-connectivity},
\[
\liminf_{k\to\infty}n_k^2\nu_2(G_k)\ge2(\Delta-2)\pi^2.
\]
Together with \eqref{eq:regular-limit}, this gives
\(2(\Delta-2)\pi^2\le m\pi^2/R\), hence \eqref{eq:even-resistance}.
\end{proof}

The next lemma extends these bounds to more general terminal degrees and
allows parallel edges at the terminals, as needed for the networks in
Section~\ref{sec:proof}.

\begin{lemma}\label{lem:completion} Let \(H\) be a connected loopless unit-resistance multigraph
with distinct terminals \(s,t\) and \(N\) other vertices. Suppose every
nonterminal has degree \(\Delta\), with no parallel edges between nonterminals.
Write \(d_s=d_H(s)\), \(d_t=d_H(t)\). For odd \(\Delta\ge3\),
\begin{equation}
R_H(s,t)\le\frac{N+d_s+d_t+\Delta+1}{\Delta-1}.                            \label{eq:odd-completion}
\end{equation}
For even \(\Delta\ge4\), if \(d_s,d_t\) are positive and even, then
\begin{equation}
R_H(s,t)\le\frac{N+d_s+d_t+2\Delta-1}{2(\Delta-2)}.                         \label{eq:even-completion}
\end{equation}
\end{lemma}

\begin{proof}
For a terminal of degree \(d\), we construct a replacement graph \(F\)
with a distinguished vertex of degree \(d_0\), exactly \(d\) other vertices
of degree \(\Delta-1\), and all remaining vertices of degree \(\Delta\).
We take \(d_0=\Delta-1\) in the odd
case, and \(d_0=2\) at \(s\), \(d_0=\Delta-2\) at \(t\) in the even case.
Whenever \(\overline F\) is specified below, take \(F\) to be its complement
on the same vertex set; thus
\(d_F(v)=|V(F)|-1-d_{\overline F}(v)\).

For odd \(\Delta\) and \(d<\Delta-1\), construct the complement \(\overline F\)
as follows. If \(d\) is odd, take a matching of \((d+1)/2\) edges and
\(\Delta-d\) isolated vertices. If \(d\) is even, take the disjoint union
of \(C_{d+1}\) and a matching of \((\Delta+1-d)/2\) edges.
For \(d\ge\Delta-1\), take \(V(F)=\mathbb Z_{d+1}\) and join each \(i\)
to \(i\pm j\), \(1\le j\le(\Delta-1)/2\), with indices modulo \(d+1\).
The three constructions have orders \(\Delta+1\), \(\Delta+2\) and
\(d+1\), respectively. By the degree formula above in the first two cases,
and by the cyclic adjacency rule in the third, each has exactly \(d+1\)
vertices of degree \(\Delta-1\), with all other vertices of degree \(\Delta\).
Choose one of the former as the distinguished vertex.

For even \(\Delta\), \(d\) is even. For \(d\le\Delta-2\), construct
\(\overline F\) at \(s\) as the disjoint union of \(K_2\) and a star
\(K_{1,\Delta-1}\) with a matching of \(d/2\) edges added among its leaves;
distinguish the star centre. At \(t\), construct \(\overline F\) as the disjoint
union of \(P_3\),
a matching of \((d-2)/2\) edges and \(\Delta-d\) isolated vertices;
distinguish the internal vertex of \(P_3\). These two constructions have
orders \(\Delta+2\) and \(\Delta+1\), respectively. The degree formula gives
the distinguished vertex degree \(2\) at \(s\) and \(\Delta-2\) at \(t\);
in each case, exactly \(d\) other vertices have degree \(\Delta-1\),
and all remaining vertices have degree \(\Delta\).
For \(d\ge\Delta\), start with the graph on \(\mathbb Z_d\) joining each
\(i\) to \(i\pm j\), \(1\le j\le(\Delta-2)/2\), and to \(i+d/2\),
with indices modulo \(d\). This graph is \((\Delta-1)\)-regular, and the
edges \(i(i+d/2)\), \(0\le i<d/2\), form a perfect matching.
Delete \(d_0/2\) edges of this matching and add the distinguished vertex
adjacent to every vertex incident with a deleted edge. Each original vertex
retains degree \(\Delta-1\), while the new vertex has degree \(d_0\).
This construction has order \(d+1\).

All the constructions are simple. In each complement construction,
\(\overline F\) is disconnected, so \(F\) is connected. The cyclic
constructions are connected because the edges \(i(i+1)\) are retained,
and the added distinguished vertex, when present, has \(d_0>0\) neighbours.
The orders given above imply \(|V(F)|\le d+\Delta\) in every case.

Replace \(s,t\) by vertex-disjoint replacement graphs \(F_s,F_t\), with
distinguished vertices \(a,b\), respectively. At each terminal, reattach its incident edges to distinct
vertices of degree \(\Delta-1\) other than the distinguished vertex. The resulting graph \(G\) is
connected and simple, and all its nonterminals have degree \(\Delta\).
Its terminal degrees are \(\Delta-1,\Delta-1\) for odd \(\Delta\), and
\(2,\Delta-2\) at \(a,b\), respectively, for even \(\Delta\).
Since \(G\) retains the \(N\) nonterminals of \(H\) and replaces \(s,t\)
by \(F_s,F_t\), its order is
\[
n:=|V(G)|=N+|V(F_s)|+|V(F_t)|\le N+d_s+d_t+2\Delta.
\]

Let
\[
\Omega=\{x\in\mathbb R^{V(G)}:x_a=0,\ x_b=1\},
\]
and let \(\Omega'\subseteq\Omega\) consist of the functions that are
constant on each of \(V(F_s)\) and \(V(F_t)\). These functions correspond bijectively, with equal energies, to the
functions on \(H\) taking values \(0\) at \(s\) and \(1\) at \(t\). Hence
\[
\frac1{R_H(s,t)}=\min_{x\in\Omega'}\mathcal E_G(x)
\ge\min_{x\in\Omega}\mathcal E_G(x)=\frac1{R_G(a,b)}.
\]
For odd \(\Delta\), applying \eqref{eq:odd-resistance} to \(G\) gives
\[
R_H(s,t)\le R_G(a,b)\le\frac{n}{\Delta-1}-1
\le\frac{N+d_s+d_t+\Delta+1}{\Delta-1},
\]
which proves \eqref{eq:odd-completion}.
For even \(\Delta\), applying \eqref{eq:even-resistance} to \(G\) gives
\[
R_H(s,t)\le R_G(a,b)\le\frac{n-1}{2(\Delta-2)}
\le\frac{N+d_s+d_t+2\Delta-1}{2(\Delta-2)},
\]
which proves \eqref{eq:even-completion}.
\end{proof}

\section{Proof of Theorem~\ref{thm:main}}\label{sec:proof}

\begin{proposition}\label{prop:comparison} For every fixed integer \(\Delta\ge3\), as \(n\to\infty\),
\begin{equation}
\Delta-\lambda_1(n,\Delta)\le\frac{c_\Delta\pi^2}{4n^2}+O_\Delta(n^{-3}).                \label{eq:comparison}
\end{equation}
\end{proposition}

\begin{proof}
For sufficiently large \(n\), choose integers \(m,r\) such that
\begin{equation}
n=(\Delta+1)m+\Delta+2+r,\qquad m\ge1,\qquad0\le r\le\Delta. \label{eq:order}
\end{equation}
Take \(G=G_{\Delta,m,r}\) from Section~\ref{sec:preliminaries}.
By the construction and \eqref{eq:order}, \(G\in\mathcal G^*(n,\Delta)\).

First let \(\Delta\) be odd. The effective resistance between the terminals of
\(L_\Delta\) is \(2/(\Delta-1)\): with boundary values \(0,1\), the energy
minimizer has value \(1/2\) at every nonterminal by symmetry and
harmonicity, and its energy is \((\Delta-1)/2\).

For even \(\Delta\ge4\), consider one copy of \(B_{\Delta,2}\) with
terminals \(a,b\). With boundary values \(0,1\), the energy minimizer is
constant on each middle clique by symmetry and uniqueness. Identifying the
vertices of each of these cliques to a single vertex and deleting the resulting
loops preserves the minimum energy. The resulting multigraph has four vertices
\(a,u,v,b\), with edge multiplicities \(2\), \(2(\Delta-2)\), and \(\Delta-2\)
on \(au,uv,vb\), respectively, and no other edges.
The series law for effective resistance~\cite{DoyleSnell} therefore gives
\[
R_{B_{\Delta,2}}(a,b)=\frac12+\frac1{2(\Delta-2)}+\frac1{\Delta-2}
=\frac{\Delta+1}{2(\Delta-2)}.
\]

Recall that \(c_\Delta=\Delta-1\) for odd \(\Delta\) and
\(c_\Delta=2(\Delta-2)\) for even \(\Delta\), and put
\(\ell=(\Delta+1)/c_\Delta\), \(T=m\ell\).
In the odd case, \(\ell=1+2/(\Delta-1)\); in the even case,
\(\ell=R_{B_{\Delta,2}}(a,b)\).

Let \(g:[0,T]\to\mathbb R\)
be given by \(g(t)=\sin(\pi t/(2T))\).
Define \(y=(y_v)_{v\in V(G)}\in\mathbb R^{V(G)}\) by setting, for \(1\le j\le m\),
\[
y_{a_j}=g((j-1)\ell+1),\qquad y_{b_j}=g(j\ell)
\]
in the odd case. In the even case, set \(y_{v_j}=g(j\ell)\) for
\(0\le j\le m\); in particular, \(y_{v_0}=0\).
On each copy of \(L_\Delta\) for odd \(\Delta\), or of \(B_{\Delta,2}\)
for even \(\Delta\), choose the nonterminal values to minimize the Dirichlet
energy subject to the prescribed terminal values. By the Dirichlet principle,
this gives the harmonic extension on that copy.
Set \(y_v=1=g(T)\) on the right cap, agreeing with the value at the right
terminal of the chain, and \(y_v=0\) on the added path vertices, if any.
All edges within the cap or joining it to the chain, as well as those within
the added path, therefore contribute zero to \eqref{eq:deficit}.
In the odd case, the degree-deficit term at \(a_1\) when \(r=0\), or the edge
joining \(a_1\) to the added path when \(r>0\), contributes exactly
\(g(1)^2=(g(1)-g(0))^2\), representing the initial unit interval.
Every other vertex of degree less than \(\Delta\) has value zero,
so the remaining degree-deficit terms vanish.

For \(0\le s<t\le T\), the fundamental theorem of calculus and the
Cauchy--Schwarz inequality give
\[
\begin{aligned}
(g(t)-g(s))^2
&=\left(\int_s^t g'(x)\,dx\right)^2\\
&\le\left(\int_s^t 1^2\,dx\right)
     \left(\int_s^t g'(x)^2\,dx\right)
=(t-s)\int_s^t g'(x)^2\,dx.
\end{aligned}
\]
Dividing by \(t-s>0\), we obtain
\((g(t)-g(s))^2/(t-s)\le\int_s^t g'(x)^2\,dx\).
In the odd case, \([0,T]\) is partitioned into the initial interval \([0,1]\),
the \(m-1\) intervals corresponding to the edges \(b_j a_{j+1}\), each of
length \(1\), and the \(m\) intervals corresponding to copies of \(L_\Delta\),
each of length \(2/(\Delta-1)\). In the even case, the intervals
\([(j-1)\ell,j\ell]\), \(1\le j\le m\), correspond to the copies of
\(B_{\Delta,2}\) and partition \([0,T]\).
If a copy \(J\), with terminals \(a,b\), corresponds to \([s,t]\), then
\(R_J(a,b)=t-s\). Let \(u\) be its harmonic function with terminal values
\(0,1\). The chosen extension is
\(y_v=g(s)+(g(t)-g(s))u_v\), so \eqref{eq:harmonic} gives
\[
\mathcal E_J(y)=(g(t)-g(s))^2\mathcal E_J(u)
=\frac{(g(t)-g(s))^2}{R_J(a,b)}
=\frac{(g(t)-g(s))^2}{t-s}.
\]
The same expression gives the contributions of the unit intervals described
above. Writing \(\mathcal P\) for this partition and summing, we obtain
\begin{equation}
\begin{aligned}
y^{\mathsf T}(\Delta I-A(G))y
&=\sum_{[s,t]\in\mathcal P}\frac{(g(t)-g(s))^2}{t-s}\\
&\le\sum_{[s,t]\in\mathcal P}\int_s^t g'(x)^2\,dx
 =\int_0^T g'(x)^2\,dx\\
&=\frac{\pi^2}{4T^2}\int_0^T\cos^2\!\left(\frac{\pi x}{2T}\right)\,dx
 =\frac{\pi^2}{8m\ell}.
\end{aligned}                                                     \label{eq:trial-energy}
\end{equation}

From each of the \(m\) copies in the chain, count all vertices when \(\Delta\)
is odd, and all vertices except its left terminal when \(\Delta\) is even.
In either case, each copy contributes \(\Delta+1\) vertices,
and the counted sets are disjoint. By the maximum principle and monotonicity
of \(g\), the values in the \(j\)-th copy lie between
\(g((j-1)\ell)\) and \(g(j\ell)\).
The function \(h(s)=\sin^2(\pi s/2)\) is increasing on \([0,1]\). Its left and
right Riemann sums on the uniform partition into \(m\) intervals differ by
\(1/m\) and bound \(\int_0^1 h(s)\,ds=1/2\). Thus the sum of \(y_v^2\) over the counted
vertices differs from \((\Delta+1)m/2\) by at most \(\Delta+1\).
The remaining \(\Delta+2+r\le2(\Delta+1)\) coordinates lie in \([0,1]\), so
\begin{equation}
\left|\|y\|_2^2-\frac{(\Delta+1)m}{2}\right|\le3(\Delta+1).              \label{eq:trial-norm}
\end{equation}
The vector is nonzero, since the cap
has value one. Equations \eqref{eq:trial-energy}--\eqref{eq:trial-norm} and the Rayleigh principle give
\[
\Delta-\lambda_1(G)\le
\frac{\pi^2}{4(\Delta+1)\ell m^2}+O_\Delta(m^{-3})
=\frac{c_\Delta\pi^2}{4n^2}+O_\Delta(n^{-3}),
\]
because \((\Delta+1)m=n-\Delta-2-r\) and \(0\le r\le\Delta\). Since this graph lies in
\(\mathcal G^*(n,\Delta)\), \eqref{eq:comparison} follows. \end{proof}

\begin{proposition}\label{prop:lower}
For every fixed integer \(\Delta\ge3\), as \(n\to\infty\),
\[
\Delta-\lambda_1(n,\Delta)\ge
\frac{c_\Delta\pi^2}{4n^2}-O_\Delta(n^{-5/2}).
\]
\end{proposition}

\begin{proof}
Let \(G\in\mathcal G^*(n,\Delta)\) satisfy \(\lambda_1(G)=\lambda_1(n,\Delta)\), and put
\(\mu=\Delta-\lambda_1(G)\). Let \(f=(f_v)_{v\in V(G)}>0\) be a Perron vector
of \(G\), normalized by \(\max_{v\in V(G)}f_v=1\).
Since \(G\) is connected and nonregular, \(\mu>0\).

Adjoin a ground vertex \(o\) and \(\Delta-d_G(v)\) parallel edges from
each \(v\in V(G)\) to \(o\), obtaining a network \(\widehat G\) with unit edge resistances.
Extend \(f\) to \(V(\widehat G)\) by setting \(f_o=0\).
The network is connected and loopless; every original vertex now
has degree \(\Delta\). From \eqref{eq:deficit}, with
\(S=\sum_{v\in V(G)}f_v^2\) and \(Q=\mathcal E_{\widehat G}(f)\),
\begin{equation}
Q=\mu S.                                                        \label{eq:ground-energy}
\end{equation}
By \eqref{eq:comparison} and \(S\le n\), \(Q\le A_\Delta/n\) for some constant \(A_\Delta>0\)
and all sufficiently large \(n\). A simple path from \(o\) to a
vertex \(v\) with \(f_v=1\) has length at most \(n\). Applying the
Cauchy--Schwarz inequality along this path gives \(1\le nQ\). Also
\(e(\widehat G)=\Delta n-e(G)\le \Delta n\). Consequently
\begin{equation}
\frac1n\le Q\le\frac{A_\Delta}{n},\qquad
\mathcal V:=\sum_{uv\in E(\widehat G)}|f_u-f_v|
\le\sqrt{\Delta nQ}\le M_\Delta:=\sqrt{\Delta A_\Delta}.                              \label{eq:variation}
\end{equation}

Fix an integer \(K\ge2\) and set \(\eta=1/(K+1)\). For \(t\in(0,1)\), let
\(b(t)\) count the edges \(uv\) of \(\widehat G\), with multiplicity, such that
\(\min\{f_u,f_v\}<t<\max\{f_u,f_v\}\).
Each edge \(uv\) contributes to this count on the open interval between
\(f_u\) and \(f_v\), and thus
\[
\int_0^1b(t)\,dt=\mathcal V.
\]
In each open interval \((i\eta,(i+1)\eta)\), \(0\le i\le K\),
choose a level \(\tau_i\) different from every vertex value such that
\(b(\tau_i)\le\eta^{-1}\int_{i\eta}^{(i+1)\eta}b(t)\,dt\).
Since \(b(t)\) is piecewise constant, such a level exists. We obtain
\begin{equation}
\begin{gathered}
\sum_{i=0}^Kb(\tau_i)\le(K+1)M_\Delta,\qquad
0<\tau_0<\eta,\qquad \tau_K>1-\eta,\\
0<\delta_i:=\tau_i-\tau_{i-1}<2\eta\qquad(1\le i\le K).
\end{gathered}                            \label{eq:levels}
\end{equation}
For every \(t\in(0,1)\) that is not a vertex value, \(b(t)>0\), since the cut separates \(o\)
from a vertex of value one in the connected network \(\widehat G\).

For \(1\le i\le K\), define \(f^{(i)}\in\mathbb R^{V(\widehat G)}\) by
\[
f^{(i)}_v=\min\{\delta_i,\max\{0,f_v-\tau_{i-1}\}\}
\qquad(v\in V(\widehat G)).
\]
For each edge \(uv\), choose the labels \(u,v\) so that \(f_u\le f_v\).
The differences \(f^{(i)}_v-f^{(i)}_u\) are nonnegative and sum to at most
\(f_v-f_u\), since the intervals \([\tau_{i-1},\tau_i]\) have disjoint interiors.
The sum of their squares is therefore at most \((f_v-f_u)^2\).
Summing over all edges gives
\begin{equation}
Q\ge\sum_{i=1}^K\mathcal E_{\widehat G}(f^{(i)}).                     \label{eq:clipping}
\end{equation}
For each \(i\), identify the vertices satisfying \(f_v\le\tau_{i-1}\) to form
a terminal \(s_i\), and those satisfying \(f_v\ge\tau_i\) to form a terminal
\(t_i\). Discard loops and retain parallel edges.
Both sets are nonempty and disjoint. Let \(H_i\) be the resulting
network and \(m_i\) its number of nonterminals. It is connected;
every nonterminal has degree \(\Delta\), and no two nonterminals are joined
by parallel edges, since they are original vertices of the simple graph \(G\).
The degrees of \(s_i,t_i\) are \(b(\tau_{i-1}),b(\tau_i)\), respectively,
including edges crossing both levels.

When \(\Delta\) is even, the degree of \(o\) is \(\Delta n-2e(G)>0\) and even,
and every other vertex of \(\widehat G\) has degree \(\Delta\). The edge boundary of any
vertex set in this network has even size, because its size is the sum of the
degrees in the set minus twice its internal edge count. Thus the terminal
degrees in \(H_i\) are positive and even. Set
\[
c=c_\Delta,\qquad \kappa=\kappa_\Delta=
\begin{cases}\Delta+1,&\Delta\text{ odd},\\2\Delta-1,&\Delta\text{ even}.
\end{cases}
\]
Applying Lemma~\ref{lem:completion} to \(H_i\) gives
\[
R_{H_i}(s_i,t_i)\le
\frac{m_i+b(\tau_{i-1})+b(\tau_i)+\kappa}{c}.
\]
The vector \(f^{(i)}\) is constant on each of the two identified vertex sets, so it
defines a vertex function on \(H_i\) with values \(0\) at \(s_i\) and
\(\delta_i\) at \(t_i\). The induced function has the same energy as
\(f^{(i)}\): edges within either set contribute zero, and all remaining
edges are retained with their multiplicities. By \eqref{eq:resistance} and
the preceding bound,
\begin{equation}
\mathcal E_{\widehat G}(f^{(i)})
\ge\frac{c\delta_i^2}{m_i+b(\tau_{i-1})+b(\tau_i)+\kappa}.             \label{eq:layer-energy}
\end{equation}

Assign each \(v\in V(G)\) to layer \(i\) if \(\tau_{i-1}<f_v<\tau_i\).
Assign the remaining vertices of \(G\) to layer 1 if \(f_v<\tau_0\), and to
layer \(K\) if \(f_v>\tau_K\). Let \(n_i\) be the number of vertices assigned
to layer \(i\), so \(n_i\ge m_i\) and \(\sum_{i=1}^K n_i=n\).
Set
\begin{equation}
\begin{gathered}
w_i=n_i+b(\tau_{i-1})+b(\tau_i)+\kappa,\qquad L=\sum_{i=1}^Kw_i,\\
n\le L\le n+2M_\Delta(K+1)+\kappa K.
\end{gathered}                                                       \label{eq:length}
\end{equation}
Each \(w_i\) is positive. The upper bound for \(L\) follows from
\eqref{eq:levels}, since each boundary count occurs at most twice.

Partition \([0,L]\) into consecutive intervals of lengths \(w_1,\ldots,w_K\).
Let \(F:[0,L]\to\mathbb R\) be continuous and linear on each interval,
with values \(\tau_{i-1}-\tau_0\) and \(\tau_i-\tau_0\) at the left and right
endpoints of the \(i\)-th interval, respectively. Thus \(F(0)=0\), and
\eqref{eq:clipping}--\eqref{eq:layer-energy} give
\begin{equation}
\int_0^L F'(x)^2\,dx=\sum_{i=1}^K\frac{\delta_i^2}{w_i}\le\frac Qc.   \label{eq:interpolation}
\end{equation}
For \(1\le i\le K\), put \(\alpha_i=\tau_{i-1}-\tau_0\), the minimum of \(F\)
on the \(i\)-th interval. Every vertex \(v\) assigned to layer \(i\)
satisfies \(f_v\ge\alpha_i\): for \(i=1\), \(\alpha_1=0\), and for \(i>1\),
\(f_v>\tau_{i-1}\). If \(i<K\), then \(f_v<\tau_i\), giving
\(f_v-\alpha_i<\delta_i+\tau_0<3\eta\). In the last layer,
\(f_v-\alpha_K\le1-\tau_{K-1}+\tau_0<3\eta\).
Since \(f_v,\alpha_i\in[0,1]\), we obtain \(f_v^2\le\alpha_i^2+6\eta\).
Also \(F\ge\alpha_i\) on the \(i\)-th interval and \(w_i\ge n_i\), so
\begin{equation}
S\le\sum_{i=1}^K n_i\alpha_i^2+6n\eta
\le\int_0^L F(x)^2\,dx+\frac{6n}{K+1}.                            \label{eq:norm-comparison}
\end{equation}

Since \(F(0)=0\), applying Wirtinger's inequality in the form
\cite[\S7.7, (7.7.1), p.~184]{HLP} to
\(u(t)=F(2Lt/\pi)\), \(0\le t\le\pi/2\), gives
\begin{equation}
\int_0^L F(x)^2\,dx\le\frac{4L^2}{\pi^2}\int_0^L F'(x)^2\,dx.      \label{eq:wirtinger}
\end{equation}

Combining \eqref{eq:ground-energy}, \eqref{eq:variation}, and \eqref{eq:interpolation}--\eqref{eq:wirtinger}, we obtain
\begin{equation}
\frac1\mu=\frac SQ
\le\frac{4L^2}{c\pi^2}+\frac{6n}{(K+1)Q}
\le\frac{4L^2}{c\pi^2}+\frac{6n^2}{K+1}.                        \label{eq:reciprocal}
\end{equation}
Take \(K=\lceil\sqrt n\rceil\). By \eqref{eq:length}, \(L/n=1+O_\Delta(n^{-1/2})\),
so substitution into \eqref{eq:reciprocal} gives
\[
\frac1{\mu n^2}\le\frac4{c_\Delta\pi^2}+O_\Delta(n^{-1/2}).
\]
Taking reciprocals and dividing by \(n^2\) gives
\[
\mu\ge\frac{c_\Delta\pi^2}{4n^2}-O_\Delta(n^{-5/2}).
\]
\end{proof}

Theorem~\ref{thm:main} follows from Propositions~\ref{prop:comparison} and~\ref{prop:lower}.

\section*{Acknowledgments}

During the preparation of this paper, OpenAI's GPT-6 was consulted for discussions of possible proof approaches and for suggestions on the exposition. These suggestions were considered and refined by the authors in developing the mathematical arguments. The authors have independently verified all proofs in the manuscript and take full responsibility for its correctness and content.

\bibliographystyle{abbrv}
\bibliography{refs}
\end{document}